\documentclass[a4paper,11pt]{amsart}
\usepackage{graphicx} % Required for inserting images
\usepackage{amsmath}
\usepackage{amssymb}
\usepackage{amsthm}
\usepackage{hyperref}
\usepackage{enumerate}
\usepackage{tikz}
\usetikzlibrary{calc,3d}
\usepackage{tikz-cd}
 \usepackage{xy}
\input xy
\xyoption{all}
\usepackage{caption, graphicx}
\usepackage[alphabetic, msc-links]{amsrefs}
\usepackage{varioref, cleveref}
\hypersetup{colorlinks=true, linkcolor=blue, citecolor=blue}
\usepackage{a4wide}
\usepackage{comment}
\usepackage[margin=1in]{geometry}
\usepackage{ytableau}
\usepackage{cleveref}

\numberwithin{equation}{section}
\numberwithin{equation}{subsection}
\usepackage{a4wide}
\usepackage{comment}
\usepackage[mathcal]{eucal}

\newtheorem{theorem}{Theorem}[section]

\newtheorem{proposition}[theorem]{Proposition}
\newtheorem{conjecture}[theorem]{Conjecture}
\newtheorem{corollary}[theorem]{Corollary}
\newtheorem{lemma}[theorem]{Lemma}
\newtheorem{claim}[theorem]{Claim}
\newtheorem{defn}[theorem]{Definition}

\newtheorem{definition}[theorem]{Definition}
\newtheorem*{remark*}{Remark}
\newtheorem{remark}[theorem]{Remark}

\newtheorem{example}[theorem]{Example}

\newtheorem{question}[theorem]{Question}

\newcommand{\Z}{\mathbb{Z}}
\newcommand{\R}{\mathbb{R}}

\numberwithin{equation}{section}
\numberwithin{figure}{section}

\DeclareMathOperator{\tr}{tr}

\DeclareMathOperator{\gl}{GL}

\title{Introduction to finite Coxeter groups \\ and their representations} 

\author{Pooja Singla }
\address{Department of Mathematics and Statistics, IIT Kanpur,  Kanpur 208016, India}
\email{psingla@iitk.ac.in}



\begin{document}

	\begin{abstract}
These notes give a short introduction to finite Coxeter groups, their classification, and some parts of their representation theory, with a focus on the infinite families. They are based on lectures delivered by the  author at the conferences Recent Trends in Group Theory at IIT Bhubaneswar and the Asian–European School in Mathematics at NEHU, Shillong. The first draft was prepared by Archita Gupta (IIT Kanpur) and Sahanawaz Sabnam (NISER Bhubaneswar), to whom the author is deeply grateful. We hope these notes will be useful both for beginners and for readers who wish to study the subject further. The author also thanks the organizers and participants of the above conferences for their support and encouragement.
\begin{comment}
We hope these notes will be useful both for beginners and for readers who wish to study the subject further. Although we have tried to be accurate, any mistakes are our own, and we would be grateful for comments and suggestions to improve the text. 
\end{comment} 
\end{abstract}
\maketitle

\section { Reflection groups and Finite Coxeter groups}

Symmetry is a fundamental and visually compelling concept in mathematics.
Consider an equilateral triangle, a square, or a regular pentagon: despite their differences, each of these figures possesses a rich symmetry structure.

\begin{center}
\begin{tikzpicture}[scale=1.2]
% Triangle
\coordinate (A) at (90:1);
\coordinate (B) at (210:1);
\coordinate (C) at (330:1);
\draw[thick] (A)--(B)--(C)--cycle;
\node at (0,-1.4) {Equilateral Triangle};

% Square
\begin{scope}[xshift=4cm]
\draw[thick] (-1,-1) -- (1,-1) -- (1,1) -- (-1,1) -- cycle;
\node at (0,-1.4) {Square};
\draw[dashed] (0,-1.2) -- (0,1.2); % reflection axis
\end{scope}

% Pentagon
\begin{scope}[xshift=8cm]
\foreach \i in {1,...,5}
  \coordinate (P\i) at (90+72*\i:1);
\draw[thick] (P1)--(P2)--(P3)--(P4)--(P5)--cycle;
\node at (0,-1.4) {Regular Pentagon};
\end{scope}
\end{tikzpicture}
\end{center}

These symmetries include both \emph{rotations} and \emph{reflections}, and in fact, reflections serve as the basic building blocks: many rotational symmetries can be expressed as the composition of two reflections.

As an example, we consider the group of symmetries, denoted $Sym^{4-gon}$, of the following square:
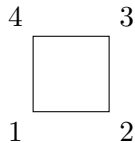
\begin{center}
\begin{tikzpicture}[scale=1]
    % Coordinates
    \coordinate (A) at (0,0);
    \coordinate (B) at (1,0);
    \coordinate (C) at (1,1);
    \coordinate (D) at (0,1);

    % Draw square
    \draw (A) -- (B) -- (C) -- (D) -- cycle;

    % Labels
    \node[below left]  at (A) {1};
    \node[below right] at (B) {2};
    \node[above right] at (C) {3};
    \node[above left]  at (D) {4};
\end{tikzpicture}
\label{figure1}
\end{center}

Then $Sym^{4-gon}$= $\langle r,s \mid r^4=1, s^2=1,srs=r^{-1} \rangle$, where $r$ denotes the rotation by $90^0$ and $s$ denotes a reflection. 
 We note that this group is also given by  $Sym^{4-gon}= \langle s,rs \mid s^2=1, (rs)^2=1, (s.rs)^4=1 \rangle$, where $s$ and $rs$ are reflections.
Thus any element of the symmetry group of a square can be written as a product of reflections.

In two dimensions, the \emph{dihedral groups} describe the symmetries of regular $n$-gons and it can be easily seen that dihedral groups are generated by reflections of $\mathbb R^2$.  
In three dimensions, the cube or dodecahedron exhibit similar reflection-generated symmetries, with reflections now taken in planes rather than lines.

\begin{center}
\begin{tikzpicture}[scale=1.3]
% Cube front face
\draw[thick] (0,0,0) -- (1,0,0) -- (1,1,0) -- (0,1,0) -- cycle;
% Cube back face
\draw[thick] (0.4,0.4,0.4) -- (1.4,0.4,0.4) -- (1.4,1.4,0.4) -- (0.4,1.4,0.4) -- cycle;
% Connect vertices
\draw[thick] (0,0,0) -- (0.4,0.4,0.4);
\draw[thick] (1,0,0) -- (1.4,0.4,0.4);
\draw[thick] (1,1,0) -- (1.4,1.4,0.4);
\draw[thick] (0,1,0) -- (0.4,1.4,0.4);
\node at (0.7,-0.3,0) {Cube};
\end{tikzpicture}
\end{center}

Higher-dimensional analogues exist, and Coxeter groups provide a systematic framework for understanding \emph{all} symmetry groups generated by reflections, regardless of dimension. Finite Coxeter groups unify several familiar classes of groups:
\begin{itemize}
    \item Dihedral groups,
    \item Symmetric groups ($A_n$),
    \item Hyperoctahedral groups ($B_n$),
    \item Symmetry groups of Platonic solids.
\end{itemize}
Beyond their geometric role, Coxeter groups appear in:
\begin{itemize}
    \item Root systems in Lie theory,
    \item Weyl groups of semisimple Lie algebras and algebraic groups,
    \item Classification problems in algebra, combinatorics, and mathematical physics.
\end{itemize}

H.~S.~M.~Coxeter~\cite{MR1503182}, in the mid 20th century, studied reflection-generated groups systematically, extending earlier work by Schläfli, Killing, and others.
His introduction of the \emph{Coxeter diagram} gave a compact graphical method for recording all relations between generators.
Today, Coxeter's framework is fundamental in diverse areas of modern mathematics.

Many beautiful and highly symmetric structures in mathematics and nature can be described purely in terms of reflections and the relations between them.
Finite Coxeter groups provide a complete classification of all such symmetry groups in the finite case, while also leading to rich and intriguing infinite examples.

These notes provide an introduction to finite Coxeter groups, their classification, and their representations, with a primary focus on the infinite families of Coxeter groups. We begin with a brief introduction to finite Coxeter groups, their geometric and algebraic definitions, and their basic properties. This is followed by a discussion of their classification, highlighting both the infinite families and the exceptional types. We primarily follow \cite{MR1066460} for this part. 

Next, we review the essential concepts from the representation theory of finite groups, providing the necessary background for readers with little prior exposure. Building on this foundation, we discuss a method to obtain irreducible representations of the infinite families of Coxeter groups. For the introduction to representation, these notes are based on \cite{MR450380} and we follow \cite{MR2857707} for discussion regarding the representations of finite Coxeter groups. 
Some other references that reader may find interesting on this topic are \cite{MR2133266}, \cite{MR2561378}, \cite{MR3445448}, and \cite{MR777684}.

\subsection{Definition of finite reflection groups and their properties} 
 We fix a finite dimensional Euclidean space $V$ with a positive definite bilinear form $\langle, \rangle$. Whenever we do not specify the Euclidean space, we assume it to be $\mathbb R^n$ with the standard inner product. 
\begin{definition}[Reflection]  For  non-zero $ \alpha \in V$, let $f:V \to V$ be a linear map such that $f(\alpha)=- \alpha$ and $f(x)=x$ for all $x \in \alpha^\perp.$ Then we say $f$ is a reflection with respect to $\alpha$ and is denoted by $s_\alpha.$
\end{definition}

Recall the group of the symmetries of the  square. 
By considering the action of rotation $r$ and reflection $s$ on vertices, we have   

$r =
\begin{pmatrix}
    1 & 2 & 3 & 4\\
    2 & 3 & 4 & 1
\end{pmatrix}$,\,\, 
$s =
\begin{pmatrix}
    1 & 2 & 3 & 4\\
    1 & 4 & 3 & 2
\end{pmatrix}$,\,\, 
$srs =
\begin{pmatrix}
    1 & 2 & 3 & 4\\
    4 & 1 & 2 & 3
\end{pmatrix}$$=r^{-1}$.\\

By considering the coordinates of the vertices in $\mathbb R^2$ as given below: 

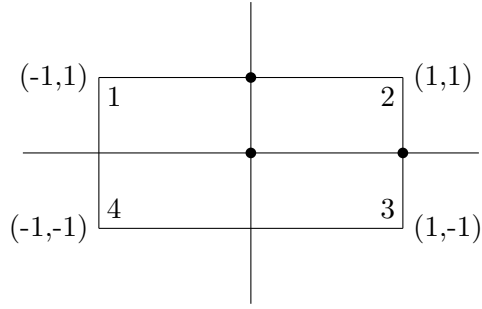
\begin{center}
   \begin{tikzpicture}
    \draw (-2,0) -- (2,0);
    \draw (-2,2) -- (2,2) ;
    \draw (0,3) -- (0,-1);
    \draw (-2,2) -- (-2,0);
    \draw (2,2) -- (2,0);
    \draw (-3,1) -- (3,1);
    \fill (-2,2) circle(0pt)node[anchor=east]{(-1,1)};
    \fill (-1.8,2) circle(0pt)node[anchor=north]{1};
    \fill (2,2) circle(0pt)node[anchor=west]{(1,1)} ;
     \fill (1.8,2) circle(0pt)node[anchor=north]{2};
     \fill (2,0) circle(0pt)node[anchor=west]{(1,-1)};
       \fill (1.8,0) circle(0pt)node[anchor=south]{3};
       \fill (-2,0) circle(0pt)node[anchor=east]{(-1,-1)};
       \fill (-1.8,0) circle(0pt)node[anchor=south]{4};
       \fill (0,2) circle(2pt);
       \fill (2,1) circle(2pt);
       \fill (0,1) circle(2pt);
       \end{tikzpicture} 
\end{center}
 we note that  $s=s_{(1,1)}$ and  $rs=s_{(1,0)}$. 
We observe that reflections are highly geometric in nature. By definition, $$s_\alpha(\lambda) = \lambda - \frac{2 \langle \lambda, \alpha\rangle }{\langle \alpha,\alpha\rangle } \alpha,$$
for any $\alpha \in V$. Also, $\langle s_\alpha (\lambda), s_\alpha (\lambda) \rangle  = \langle \lambda, \lambda 
\rangle $ for all $\alpha, \lambda \in V$. Therefore $s_\alpha \in O(V)$ for all $\alpha \in V$, where $O(V)$ is the orthogonal group associated to $(V, \langle, \rangle)$. Hence any finite Coxeter group is a subgroup of an orthogonal group.

\begin{definition}[Reflection group] A finite group generated by reflections is called reflection group.
\end{definition}
We give a few examples of reflection groups. 
\begin{example}
$Sym^{n-gon}= \langle s,rs \mid s^2=1, (rs)^2=1, (srs)^n=1 \rangle$ is a finite reflection group generated by two reflections of $\mathbb R^2$. 
\end{example} 

\begin{example}
The symmetric group $S_n$: For this consider the standard basis $\{\epsilon_i \mid i\}$ of $\R^n$. Then $S_n$ acts on $\R^n$ by permuting these basis vectors. Note that $(i,j)$ here corresponds to reflection $s_{\epsilon_i - \epsilon_j}$. Therefore the group  $S_n$ is a finite reflection group. The group $S_n$ has the following presentation: 
$$\langle (i,i+1) \mid (i,i+1)^2=1, ((i,i+1)(i+1,i+2))^3=1, ((i,i+1)(j,j+1))^2=1 \text{ for } |i-j| > 1\rangle.$$ 

\end{example}

\begin{definition} (Essential action) A finite reflection subgroup $G$ of $O(V)$ is called essential with respect to $V$ if $V$ does not have non-trivial fixed point with respect to the $G$-action. 
	
	\end{definition} 

The group $S_n$ in fact fixes the one dimensional subspace $W \subseteq \mathbb R^n$ generated by $\epsilon_1 + \epsilon_2 + \ldots + \epsilon_n$. By considering the action of $S_n$ on $V/W \cong \mathbb R^{n-1}$, we obtain an essential action. We will always consider the action of $S_n$ on this quotient space. 

\begin{example} (Hyperoctahedral groups) The hyperoctahedral group $B_n$ for $n \geq 2$ is defined by  $B_n = (\Z/2\Z)^n \rtimes S_n$
with action of $S_n$ on $(\Z/2\Z)^n$ given by 
\[
(a_1,\ldots,a_n)^{\sigma} =  (a_{\sigma(1)},\ldots,a_{\sigma(n)}).
\]
Hence elements of $B_n$ are the form  $((a_1,\ldots,a_n), \sigma )$ with $a_i \in \{-1,1\}$ and $ \sigma \in S_n$.
This group acts on $R^n$ as follows: 
\begin{itemize} 
\item $(a_1, a_2, \cdots, a_n) (v_1, v_2, \cdots, v_n) = (a_1 v_1, \cdots, a_n v_n)$
\item $\sigma(v_1, \cdots, v_n) = (v_{\sigma(1)}, \cdots, v_{\sigma(n)}).$
\end{itemize} 
The group $B_n$ is generated by the permutation matrices and the diagonal matrices $d_i$, where all diagonal entries are equal to $1$ except the $i^{th}$ entry and the $i^{th}$ entry which is $-1.$ We note that this group is generated by reflections and hence is a finite reflection group. We also observe that the action of $B_n$ on $\mathbb R^n$ is essential.  
	
\end{example}

\begin{proposition}
\label{prop:conjugation-of-reflection}
	If $t \in O(V)$ and $\alpha \in V$ then $t s_\alpha t^{-1} = s_{t(\alpha)}$. 
	\end{proposition} 
   For any non-zero vector $\alpha \in V$, let $\mathbb R\alpha $ denote  the one dimensional space generated by $\alpha$, we shall call this to be a line passing through $\alpha.$  We note that the reflection group $W$ permutes the set  $\{ \mathbb R\alpha \mid s_\alpha \in W \}$ by \Cref{prop:conjugation-of-reflection}.  Therefore, the group $W$ acts on the following set: 
   $$\Phi= \{ \alpha \in V \mid s_{\alpha} \in W, || \alpha ||=1 \}.$$  

We note that $\Phi$ satisfies the following properties:
\begin{enumerate}
	\item[R1:] $\Phi$ is a finite set. 
	\item[R2:] $\Phi \cap \mathbb R \alpha = \{ \pm \alpha\}$ for all $\alpha \in \Phi$.
	\item[R3:] $s_\alpha(\Phi) = \Phi$ for all $\alpha \in \Phi$. 
\end{enumerate}
Any finite subset $\Phi \subset V$ satisfying the properties $(R1)-(R3)$ is called a {\bf Root system} of $V$. At times, we may also insist that all vectors of $\Phi$ are unit vectors. The following highlights the connection between the reflection groups and Root systems. 
\begin{proposition}
	Any reflection group $W$ determines a root system $\Phi$. Conversely, given a root system $\Phi$, the group generated by $\{s_\alpha \mid \alpha \in \Phi\}$ is a finite reflection group.  
\end{proposition}

\begin{comment} 
Let $W$ be a finite reflection group. By the above result,  $W$ permutes $L_\alpha$ for all $\alpha$ such that $s_\alpha \in W$. 
We can pick pair of norm one vectors $\alpha, -\alpha \in L_\alpha$ for each $s_\alpha \in W$ and consider the following set: 
 $$\Phi  = \{ \pm \alpha \mid |\alpha| = 1, s_\alpha \in W    \}.$$
 Then $\Phi$ satisfies the following
\begin{itemize}
\item[(a)] $\Phi \cap L_\alpha = \{  \pm \alpha\}, \forall \,\, s_\alpha \in W $
\item[(b)] $s_\alpha(\Phi) = \Phi$ 
\end{itemize} 
 Infact, we can remove norm one condition and focus more on stability condition. This motivates the following definition of {\bf Root System}. 

\begin{definition} Let $\Phi$ be any finite set of non-zero vectors of $V$ satisfying the following:
\begin{itemize} 
    \item[(R1)] $\Phi \cap L_\alpha = \{ \pm \alpha \}$, $\alpha \in \Phi$
\item[(R2)] $s_\alpha(\Phi ) = \Phi$ for all $\alpha \in \Phi$. 
\end{itemize} 
Then $\Phi$ is called a root system. 
\end{definition} 
For a root system $\Phi$, define $W_{\Phi} = \langle s_\alpha \mid \alpha \in \Phi \rangle$. 

\begin{claim}
	Let $\Phi$ be a root system and $W = W_{\Phi}$ be as defined above. The $W$ is a finite reflection group.
\end{claim} 
\end{comment} 
\begin{proof}
	Let $X = Span(\Phi)$. By definition, $W \subseteq O(V)$. We have a map $W \rightarrow \mathrm{Perm}(\Phi)$. We claim that this is an injective homomorphism. Let $w \in W$ such that $w(\alpha ) = \alpha$ for all $\alpha \in \Phi$. Since $W$ fixes all the vectors orthogonal to $\mathrm{Span}(\Phi)$, we obtain that $w$ fixes all vectors of $V$. Therefore $w = e$. This implies that map is injective.   
	
	\end{proof} 
 Thus, any finite reflection group is of the form $W_\Phi$ for some root system $\Phi$. 
 \begin{remark}
 There may be more than one root systems giving the same reflection group. 
\end{remark} 
\begin{example}
	For $Sym^{4-gon}$, a root system is given by $\Phi = \{ (0,\pm 1), (\pm 1, \pm 1) (\pm 1, 0)  \}$. These are not unit vectors but still satisfy $R1-R3$. 
\end{example}

\subsection{Definition of finite Coxeter groups and their properties}
\begin{definition}[Coxeter Group]
\label{def:Coxeter-group}
A pair $(W,S)$ consisting of a group $W$ and its set of generators $S \subseteq W$ with only relations of the form $(ss')^{m(s,s')}=1,$ where  $m(s,s')$ satisfy the following:  
\begin{enumerate}
	\item $m(s,s)=1$  for all $s \in S$ 
	\item  $m(s,s')=m(s',s) \ge 2$ for $s \neq s'$ in $S,$
	\item  $m(s,s')=\infty$ if no relation between $s$ and $s'.$ 
	\end{enumerate}
    is called a Coxeter system. The group $W$ with a Coxeter system $(W, S)$ for some $S$ is called a Coxeter group. 
    \end{definition}
The cardinality of the set $S$ is called to be the {\bf rank of $(W, S).$} For a Coxeter group $W,$ a presentation of the form
\begin{eqnarray}
\label{eq:Coxeter-presentation}
W & = & \langle s \mid (ss')^{m(s, s')} = 1 
 \rangle, 
 \end{eqnarray}
 where $m(s, s')$ satisfy the conditions of the Definition~\ref{def:Coxeter-group} is called a Coxeter presentation of the Coxeter group $W$. 

\begin{example}
   The group $Sym^{n- gon}$ is a coxter group because it has the following presentation:
   \[
   Sym^{n-gon}= \langle s,rs \mid s^2=1, (rs)^2=1, (srs)^n=1 \rangle. 
   \]
    \end{example}
    
    \begin{example}[The symmetric group $S_n$]
    \label{exam:Sn-Coxeter}
  The group $S_n$ is a Coxeter group, because it has the following presentation: $$S_n=\langle x_1,\ldots,x_{n-1} \mid x_i^2=1, (x_ix_j)^3=1 \text{ for } |i-j|=1, (x_ix_j)^2=1 \text{ for } |i-j| > 1 \rangle.$$
    \end{example}
    \begin{example}[Hyperoctahedral group $B_n$]
    \label{exam:Bn}
The group $B_n$ has the following presentation: 

$B_n:= \langle x_0,x_1,\ldots,x_{n-1} \mid x_i^2=1, (x_ix_j)^3=1 \text{ for } |i-j|=1 \text{ and } i,j \ge 1, (x_ix_j)^2=1 \text{ for } |i-j| > 1 \text{ and } i,j \ge 1 , (x_0x_1)^4=1 \rangle$.\\

Hence $B_n$ is a Coxeter group. 
 \end{example}
 
We now define the notion of the Coxeter matrix and Coxeter graph associated to a Coxeter group.  

\begin{definition}(Coxeter matrix)  For a Coxeter system $(W, S)$ with presentation as given in \Cref{eq:Coxeter-presentation}, we define a symmetric $|S| \times |S|$ matrix  $ M$ such that $M= (m_{s,s'})$, where $m(s, s')$ are as appeared in the Coxeter presentation \Cref{eq:Coxeter-presentation}. Then the matrix $M$ is called the Coxeter matrix of $(W, S)$. 
\end{definition}
\begin{definition}(Coxeter Graph) 
An undirected graph with vertices $\{e_s \mid s \in S\}$ and edges given by  
\begin{itemize}
\item There exists an edge from $e_s$ to $e_{s'}$ iff $m(s, s') \geq 3.$
\item The edge from $e_s$ to $e_{s'}$ is labeled by $m(s, s')$ if and only if  $m(s, s') \geq 4.$
\end{itemize}
is called a Coxeter graph of $(W, S)$. 
\end{definition}
Whenever the Coxeter presentation of the Coxeter group $W$ is already fixed, we call the obtained Coxeter matrix and Coxeter graph to be the Coxeter matrix and graph of $W$.
\begin{example}(Coxeter graph and Coxeter matrix of $S_n$) From the presentation of $S_n$ as appeared in Example~\ref{exam:Sn-Coxeter}, we obtain the  the following as Coxeter diagram and Coxeter matrix of $S_n.$
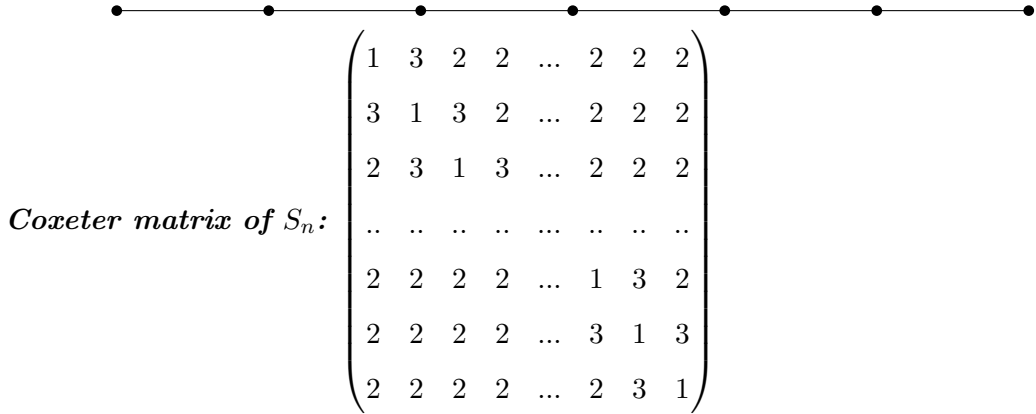
\begin{center}
    \begin{tikzpicture}
\draw (-6,0) -- (6,0);
    \fill (-6,0) circle(2pt);
    \fill (-4,0) circle(2pt);
    \fill (-2,0) circle(2pt);
    \fill (0,0) circle(2pt);
    \fill (2,0) circle(2pt);
    \fill (4,0) circle(2pt);
    \fill (6,0) circle(2pt);
\end{tikzpicture}
\label{Coxter graph of $S_{n+1}$}  
\end{center}

\textbf{Coxeter matrix of $S_n$:}
$\begin{pmatrix}
   1 & 3 & 2 & 2 & ... & 2 & 2 & 2\\
   3 & 1 & 3 & 2 & ... & 2 & 2 & 2\\
   2 & 3 & 1 & 3 & ... & 2 & 2 & 2\\
   ..  & .. & .. & .. & ... & .. & .. & ..\\
   2 & 2 & 2 & 2 & . ..& 1 & 3 & 2\\
   2 & 2 & 2 & 2 &  ...& 3 & 1 & 3\\
   2 & 2 & 2 & 2 &.. . & 2 & 3 & 1\\
\end{pmatrix}$
\end{example}
 \begin{example}[Coxeter graph of $B_n$] For $B_n$, with presentation as given in \Cref{exam:Bn}, we obtain the following as Coxeter  graph of $B_n:$ 

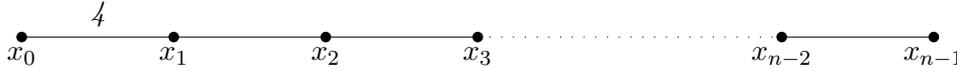
\begin{tikzpicture}
 \draw (-4,0) -- (0,0); 
   \draw (-6,0) -- (-4,0)node[midway,above]{4};
   \draw[loosely dotted] (0,0) -- (4,0);
   \draw (4,0) -- (6,0); 
   \fill (-6,0) circle(2pt)node[anchor=north]{$x_0$};
    \fill (-4,0) circle(2pt)node[anchor=north]{$x_1$};
     \fill (-2,0) circle(2pt)node[anchor=north]{$x_2$};
      \fill (0,0) circle(2pt)node[anchor=north]{$x_3$};
       \fill (4,0) circle(2pt)node[anchor=north]{$x_{n-2}$};
        \fill (6,0) circle(2pt)node[anchor=north]{$x_{n-1}$};
\end{tikzpicture}\\
\end{example} 

\subsection{Relation between finite reflection groups and finite Coxeter groups} 

In this section, we prove that every reflection group is a finite Coxeter group and conversely every finite Coxeter group is a reflection group. We will use this to classify all finite Coxeter groups.

We first discuss the ideas towards proving that a finite reflection group is a finite Coxeter group. The properties of the Root system associated to the Coxeter group can be studied independently and these lead to proving that any finite reflection group is a finite Coxeter group as follows: 

\begin{theorem}  Let $W$ be a finite reflection group with a root system $\Phi$. 
	  \begin{enumerate} 
		\item There exists a subset $\Delta$ of $\Phi$ such that $\Delta$ is a basis of $V$ and every $\alpha \in \Phi$ is either a non-negative linear combination of elements of $\Delta$ or a non-positive linear combination of elements of $\Delta.$
		\item 	The group $W$ is a finite Coxeter group generated by the set $\{s_\alpha \}_{\alpha \in \Delta}$ subject only to the relations
		$$(s_\alpha s_\beta)^{m_{\alpha,\beta}}=1; \hspace{0.5cm} \alpha, \beta \in \Delta.$$ 
		\end{enumerate} 
\end{theorem}

The set $\Delta$ in the above result is called a base of the root system $\Phi.$ For the finite reflection group $S_n= \langle (i,j) \mid i \neq j \rangle$, we have
$\Phi= \{ \epsilon_i - \epsilon_j \mid i \neq j \}$ and 
$\Delta=\{ \epsilon_i- \epsilon_{i+1} \mid 1 \le i \le {n-1}\}.$

We now proceed to prove that  every finite Coxeter group is a  reflection groups. The definition of a finite Coxeter group is as an abstract group. To show that $(W, S)$ is a finite reflection group, we first define a Euclidean space as follows: 

Define a finite dimensional real vector space $V$ by $V=span_\R\{\alpha_s \mid s \in S\}$.  Define a bilinear form on $V$ by 

\begin{eqnarray} 
	\label{eqn-Coxeter-bilinear}
B: V \times V \to \R; \,\,    \,\, B(\alpha_s, \alpha_s')  = -cos\frac{\pi}{m(s,s')}.
\end{eqnarray} 
We note that $B$ is a symmetric and non-degenerate bilinear form on $V$. Define $\sigma_s: V \to V$ by $$\sigma_s(v)= v - 2B(\alpha_s, v)\alpha_s.$$ Then $\sigma_s^2=1$, $\sigma_s \in \mathrm{GL}(V)$ and $B(\sigma_s v ,\sigma_s w )=B(v ,w )$ for all $s \in S$ and $v, w \in V$.  The following result shows that any finite Coxeter group is a finite reflection group. 
\begin{theorem}
	There exists a unique `faithful' homomorphism $\sigma: W \to \mathrm{GL}(V)$ sending $s \to \sigma_s$ such that $\sigma(W)$ is a reflection group.
\end{theorem}

We have now proved that every finite Coxeter group is a finite reflection group and conversely. Now on wards, we will use these terms interchangeably. In the next section, we will prove the classification of finite Coxeter groups. 
\subsection{Classification of finite Coxeter groups}
We now proceed to prove the classification of finite Coxeter groups $(W, S)$. We first consider the rank two case. In this case, the only possible Coxeter matrices are  $\begin{pmatrix}
	1 & k\\
	k & 1
\end{pmatrix}$ for $2 \leq k$ with Coxeter graphs given by 
\hspace{5cm}
\begin{tikzpicture}
    \fill (0,0);  \fill (0.5,0) circle(2pt);
    \fill (2.0,0) circle(2pt);
    \end{tikzpicture} \hspace{1cm} for $k = 2$\\

$\vspace{0.5cm}
 \begin{tikzpicture}   
    \fill (0,-1.5) circle(2pt);
    \fill (1.5,-1.5) circle(2pt);
     \draw(0,-1.5) -- (1.5,-1.5)node[midway,above]{k};
 \end{tikzpicture}$  \hspace{1cm} for $k \geq 3$. 
 
 In this first case, that is $k = 2$, we obtain the reflection group $\Z/2\Z \times \Z/2\Z$ and for $k \geq 3$, we obtain the group $I_2(k)$. 
 
 We note that if the Coxeter graph of a Coxeter group $(W, S)$ is disconnected then there exists subgroups $(W_1, S_1)$ and $(W_2, S_2)$ such that $W \cong W_1 \times W_2$ and $W_{i}$'s are Coxeter groups. Hence for the purpose of the classification of the finite Coxeter groups, it is sufficient to assume that the Coxeter graph of the given Coxeter group is connected. A Coxeter group with a connected Coxeter graph is called an {\it irreducible Coxeter group}. 
 
 Let $W$ be an irreducible Coxeter group. We consider the matrix $A = (a_{s, s'})$, where $a_{s, s'} = - \cos (\pi/m(s, s'))$. This is the matrix associated with the Bilinear form defined in \Cref{eqn-Coxeter-bilinear}. We will call this associate matrix of the Coxeter group. 
 
 Recall that a matrix $A \in M_n(\mathbb R)$ is called positive definite if $x^tAx >0$ for all $x$ and is called positive semi-definite if $x^tAx \ge 0$ for all $x$. We note that $A$ is positive definite (semi-definite) if and only if  all eigenvalues are real and positive (non-negative). This is equivalent to all principle minors are positive (non-negative).

 The following result summarizes the necessary property of the associated matrix $A$ of the Coxeter group.  
 \begin{theorem}
     The associated matrix of a finite Coxeter group is positive definite. 
 \end{theorem}
 This follows because the associated matrix represents standard inner product with respect to the basis $\Delta$. 
\begin{definition}(positive definite Coxeter graph) We say a Coxeter graph is positive definite if the associated matrix is positive definite. 
	
	\end{definition} 
	We have proved above that a finite Coxeter group $(W, S)$ is positive definite. As a first step towards the classification, we proceed to understand the positive definite graphs.  For this we define the notion of a subgraph. 
	\begin{definition} (subgraph) Let $\Gamma$ be a graph. A graph obtained by either removing certain vertices and the adjacent edges or obtained by decreasing certain labels of $\Gamma$, or a combination of both will be called a subgraph of $\Gamma$. 
	\end{definition}
	The following result is very useful in completing the classification. 
	
	\begin{theorem}
		Any subgraph of a positive type connected Coxeter graph is positive definite.
	\end{theorem}
Consider the family of all graphs given below. We call this to be the family of $\mathrm{Dynkin}^+$ graphs.

$A_n(n \ge 1)$ \hspace{0.4cm}
\begin{tikzpicture}
	\draw (0,0) -- (2,0);
	\draw[loosely dotted] (2,0) -- (3,0);
	\draw (3,0) -- (4,0);
	\fill (0,0) circle(2pt);
	\fill (0,0) circle(2pt);
	\fill (1,0) circle(2pt);
	\fill (2,0) circle(2pt);
	\fill (3,0) circle(2pt);
	\fill (4,0) circle(2pt);
\end{tikzpicture}
\vspace{0.3cm}

$B_n(n \ge 2)$ \hspace{0.4cm}
\begin{tikzpicture}
	\draw (-2,0) -- (0,0);
	\draw[loosely dotted] (0,0) -- (1,0);
	\draw (1,0) -- (2,0);
	\draw (2,0) -- (3,0)node[midway,above]{$4$};
	\fill (0,0) circle(2pt);
	\fill (-2,0) circle(2pt);
	\fill (-1,0) circle(2pt);
	\fill (1,0) circle(2pt);
	\fill (2,0) circle(2pt);
	\fill (3,0) circle(2pt);
\end{tikzpicture}
\vspace{0.3cm}

$D_n(n \ge 4)$ \hspace{0.7cm}
\begin{tikzpicture}
	\draw (-2,0) -- (0,0);
	\draw[loosely dotted] (0,0) -- (1,0);
	\draw (1,0) -- (2,0);
	\draw (2,0) -- (3,-0.5);
	\draw (2,0) -- (3,0.5);
	\fill (0,0) circle(2pt);
	\fill (-2,0) circle(2pt);
	\fill (-1,0) circle(2pt);
	\fill (1,0) circle(2pt);
	\fill (2,0) circle(2pt);
	\fill (3,-0.5) circle(2pt);
	\fill (3,0.5) circle(2pt);
\end{tikzpicture}
\vspace{0.3cm}

$E_6$
\hspace{2.3cm}
\begin{tikzpicture}
	\draw (0,0) -- (4,0);
	\draw (2,0) -- (2,1);
	\fill (0,0) circle(2pt);
	\fill (1,0) circle(2pt);
	\fill (2,0) circle(2pt);
	\fill (2,1) circle(2pt);
	\fill (3,0) circle(2pt);
	\fill (4,0) circle(2pt);
\end{tikzpicture}
\vspace{0.3cm}

$E_7$
\hspace{2.3cm}
\begin{tikzpicture}
	\draw (0,0) -- (5,0);
	\draw (2,0) -- (2,1);
	\fill (0,0) circle(2pt);
	\fill (1,0) circle(2pt);
	\fill (2,0) circle(2pt);
	\fill (2,1) circle(2pt);
	\fill (3,0) circle(2pt);
	\fill (4,0) circle(2pt);
	\fill (5,0) circle(2pt);
\end{tikzpicture}

\vspace{0.3cm}
$E_8$
\hspace{2.3cm}
\begin{tikzpicture}
	\draw (0,0) -- (6,0);
	\draw (2,0) -- (2,1);
	\fill (0,0) circle(2pt);
	\fill (1,0) circle(2pt);
	\fill (2,0) circle(2pt);
	\fill (2,1) circle(2pt);
	\fill (3,0) circle(2pt);
	\fill (4,0) circle(2pt);
	\fill (5,0) circle(2pt);
	\fill (6,0) circle(2pt);
\end{tikzpicture}
\vspace{0.3cm}

$F_4$
\hspace{2.3cm}
\begin{tikzpicture}
	\draw (0,0) -- (1,0);
	\draw (1,0) -- (2,0) node [midway,above]{$4$};
	\draw (2,0) -- (3,0);
	\fill (0,0) circle(2pt);
	\fill (1,0) circle(2pt);
	\fill (2,0) circle(2pt);
	\fill (3,0) circle(2pt);
\end{tikzpicture}
\vspace{0.3cm} 

$H_3$
\hspace{2.3cm}
\begin{tikzpicture}
	\draw (0,0) -- (1,0)node[midway,above]{$5$};
	\draw (1,0) -- (2,0) ;
	\fill (0,0) circle(2pt);
	\fill (1,0) circle(2pt);
	\fill (2,0) circle(2pt);
\end{tikzpicture}
\vspace{0.3cm}

$H_4$
\hspace{2.3cm}
\begin{tikzpicture}
	\draw (0,0) -- (1,0)node[midway,above]{$5$};
	\draw (1,0) -- (3,0) ;
	\fill (0,0) circle(2pt);
	\fill (1,0) circle(2pt);
	\fill (2,0) circle(2pt);
	\fill (3,0) circle(2pt);
\end{tikzpicture}
\vspace{0.3cm}

$I_2(m)$
\hspace{2.3cm}
\begin{tikzpicture}
	\draw (0,0) -- (1,0)node [midway,above]{$m$};
	\fill (0,0) circle(2pt);
	\fill (1,0) circle(2pt);
\end{tikzpicture}
\vspace{0.3cm}

\begin{theorem}
\begin{enumerate} 
\item The $\mathrm{Dynkin}^+$ family is closed under taking the subgraphs. 
\item All graphs in the $\mathrm{Dynkin}^+$  family are positive definite. 
\end{enumerate}  
\end{theorem}

We now give a family of non-positive graphs. This family is parallel to $\mathrm{Dynkin}^+$  family, and can be easily checked that the determinant of the associated matrices is zero, and hence they are not positive definite.

\begin{theorem}
The following graphs are  not positive definite.

  \vspace{0.7cm}
$\tilde{A_1}$
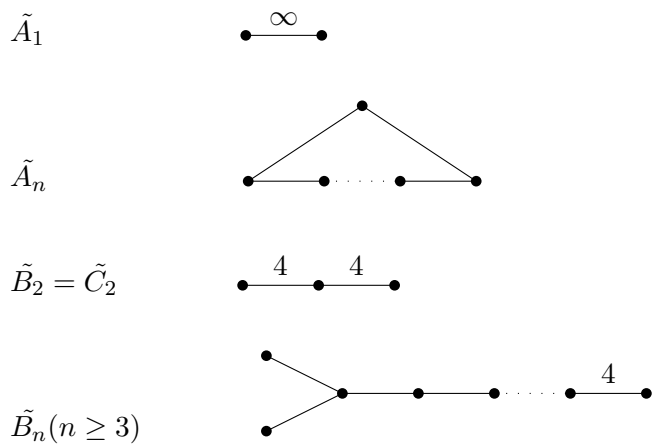
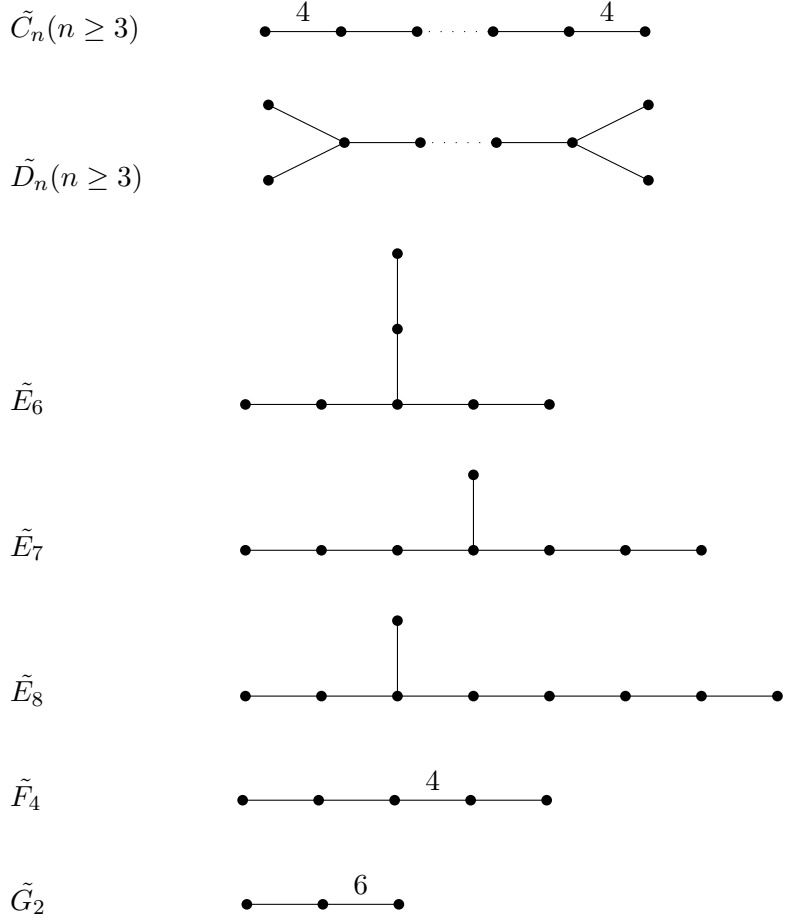
\hspace{2.3cm}
   \begin{tikzpicture}
 \draw (0,0) -- (1,0)node [midway,above]{$ \infty$};
      \fill (0,0) circle(2pt);
    \fill (1,0) circle(2pt);
    \end{tikzpicture}
  \vspace{0.7cm} 

$\tilde{A_n}$
  \hspace{2.3cm}
   \begin{tikzpicture}
   \draw (0,0) -- (1,0);
   \draw[loosely dotted] (1,0) -- (2,0);
 \draw (2,0) -- (3,0) ;
 \draw (0,0) -- (1.5,1);
 \draw (3,0) -- (1.5,1);
  \fill (0,0) circle(2pt);
    \fill (1,0) circle(2pt);
    \fill (2,0) circle(2pt);
   \fill (3,0) circle(2pt);
   \fill (1.5,1) circle(2pt);
     \end{tikzpicture} 
  \vspace{0.7cm}

$\tilde{B_2}=\tilde{C_2}$
  \hspace{1.3cm}
   \begin{tikzpicture}
   \draw (0,0) -- (1,0)node[midway,above]{$ 4$};
   \draw (1,0) -- (2,0)node[midway,above]{$ 4$};
 \fill (0,0) circle(2pt);
    \fill (1,0) circle(2pt);
    \fill (2,0) circle(2pt);
   \end{tikzpicture} 
  \vspace{0.7cm}

$\tilde{B_n}(n \ge 3)$
  \hspace{1.3cm}
   \begin{tikzpicture}
   \draw (0,0) -- (1,0);
   \draw (1,0) -- (2,0);
   \draw[loosely dotted] (2,0) -- (3,0);
   \draw (3,0) -- (4,0)node[midway,above]{$ 4$};
   \draw (0,0) -- (-1,-0.5);
   \draw (0,0) -- (-1,0.5);
   \fill (0,0) circle(2pt);
    \fill (1,0) circle(2pt);
    \fill (2,0) circle(2pt);
    \fill (3,0) circle(2pt);
    \fill (4,0) circle(2pt);
    \fill (-1,0.5) circle(2pt);
    \fill (-1,-0.5) circle(2pt);
   \end{tikzpicture}
  \vspace{0.7cm}

  $\tilde{C_n}(n \ge 3)$ 
   \hspace{1.3cm}
\begin{tikzpicture}
 \draw (-2,0) -- (-1,0)node[midway,above]{$4$};
  \draw (-1,0) -- (0,0);
     \draw[loosely dotted] (0,0) -- (1,0);
     \draw (1,0) -- (2,0);
     \draw (2,0) -- (3,0)node[midway,above]{$4$};
     \fill (0,0) circle(2pt);
    \fill (-2,0) circle(2pt);
    \fill (-1,0) circle(2pt);
     \fill (3,0) circle(2pt);
     \fill (2,0) circle(2pt);
     \fill (1,0) circle(2pt);
    \end{tikzpicture}
  \vspace{0.7cm}

$\tilde{D_n}(n \ge 3)$
\hspace{1.3cm}
   \begin{tikzpicture}
   \draw (0,0) -- (1,0);
   \draw (1,0) -- (2,0)[loosely dotted];
   \draw (2,0) -- (3,0);
   \draw (3,0) -- (4,-0.5);
   \draw (3,0) -- (4,0.5);
   \draw (0,0) -- (-1,-0.5);
   \draw (0,0) -- (-1,0.5);
   \fill (0,0) circle(2pt);
    \fill (1,0) circle(2pt);
    \fill (2,0) circle(2pt);
    \fill (3,0) circle(2pt);
    \fill (4,0.5) circle(2pt);
    \fill (4,-0.5) circle(2pt);
    \fill (-1,0.5) circle(2pt);
    \fill (-1,-0.5) circle(2pt);
   \end{tikzpicture}
  \vspace{0.7cm}

$\tilde{E_6}$
\hspace{2.3cm}
   \begin{tikzpicture}
 \draw (0,0) -- (4,0);
      \draw (2,0) -- (2,2);
    \fill (0,0) circle(2pt);
    \fill (1,0) circle(2pt);
    \fill (2,0) circle(2pt);
    \fill (2,1) circle(2pt);
    \fill (3,0) circle(2pt);
     \fill (4,0) circle(2pt);
      \fill (2,2) circle(2pt);
 \end{tikzpicture}
  \vspace{0.7cm}

$\tilde{E_7}$
\hspace{2.3cm}
   \begin{tikzpicture}
 \draw (-1,0) -- (5,0);
      \draw (2,0) -- (2,1);
      \fill (0,0) circle(2pt);
    \fill (-1,0) circle(2pt);
    \fill (1,0) circle(2pt);
    \fill (2,0) circle(2pt);
    \fill (2,1) circle(2pt);
    \fill (3,0) circle(2pt);
     \fill (4,0) circle(2pt);
     \fill (5,0) circle(2pt);
 \end{tikzpicture}
  \vspace{0.7cm}

$\tilde{E_8}$
\hspace{2.3cm}
   \begin{tikzpicture}
 \draw (0,0) -- (7,0);
      \draw (2,0) -- (2,1);
    \fill (0,0) circle(2pt);
    \fill (1,0) circle(2pt);
    \fill (2,0) circle(2pt);
    \fill (2,1) circle(2pt);
    \fill (3,0) circle(2pt);
     \fill (4,0) circle(2pt);
     \fill (5,0) circle(2pt);
     \fill (6,0) circle(2pt);
     \fill (7,0) circle(2pt);
 \end{tikzpicture}
  \vspace{0.7cm}

$\tilde{F_4}$
\hspace{2.3cm}
   \begin{tikzpicture}
 \draw (0,0) -- (2,0);
  \draw (2,0) -- (3,0) node [midway,above]{$4$};
   \draw (3,0) -- (4,0);
      \fill (0,0) circle(2pt);
    \fill (1,0) circle(2pt);
    \fill (2,0) circle(2pt);
   \fill (3,0) circle(2pt);
    \fill (4,0) circle(2pt);
     \end{tikzpicture}
  \vspace{0.7cm}

$\tilde{G_2}$
 \hspace{2.3cm}
   \begin{tikzpicture}
 \draw (0,0) -- (1,0);
  \draw (1,0) -- (2,0) node [midway,above]{$6$};
  \fill (0,0) circle(2pt);
    \fill (1,0) circle(2pt);
    \fill (2,0) circle(2pt);
  \end{tikzpicture}
  \vspace{0.7cm}  
\end{theorem}
We are now in a position to state the classification result for all irreducible Coxeter groups. 
\begin{theorem}
	Let $G$ be any irreducible Coxeter group. Then the Coxeter graph of $G$ is a $\mathrm{Dynkin}^+$  graph. Conversely, for every $\mathrm{Dynkin}^+$  graph $\Gamma$ there exists a Coxeter group $W$ with Coxeter graph $\Gamma$.  
\end{theorem}
The Dynkin diagram $A_n$ corresponds to the Symmetric group $S_{n+1}$, $B_n$ corresponds to the Hyperoctahedral groups and $I_2(m)$ correspond to $\mathrm{Sym}^{m-gon}.$ We now give Coxeter groups corresponding to the Dynkin diagram $D_n$. For Coxeter groups corresponding to remaining Dynking  diagrams, i.e. $E_6, E_7, E_8, F_4, H_3$ and $H_4$, see \cite{MR1066460}. 
\subsection{Groups of type $D_n$}
We have 
\[
(\mathbb{Z}/2\mathbb{Z})^n \cong \{\left(\begin{smallmatrix}
    a_1 & 0 & \dots & 0\\
    0 & a_2 & \dots & 0\\
    \dots & \dots & \dots & \dots\\
    0 & 0 & \dots & a_n
\end{smallmatrix}\right) \mid a_i\in \{\pm 1\}\}
\]
Define a map $\mathrm{det} \colon (\mathbb{Z}/2\mathbb{Z})^n \longrightarrow \{\pm 1\}$. Let $K=\mathrm{Ker}(\mathrm{det})$. 
\bigskip
\begin{definition}[The group $D_n$] For an integer $n\geq 4$, the hyperoctahedral group $D_n$ is defined as the index $2$ subgroup of $B_n$ defined by $K \rtimes S_n$. A presentation of this group is as follows:

\begin{tabular}{ccc}
$D_n$ & = & $\langle x_1,x_2,\ldots,x_n \mid x_i^2=1,\,\, 1 \le i \le n ;\,\, (x_ix_{i+1})^3=1,\,\,  2 \le i \le n-1;$  \\ 
& & $(x_ix_j)^2=1, \,\, |i-j| \neq 1 \text{ and } (i,j) \neq (1,3);  \,\, (x_1x_2)^2=(x_1x_3)^3=1 \rangle.$
\end{tabular} 

It is easy to that a Coxeter graph of $D_n$ with the above presentation is as follows: 

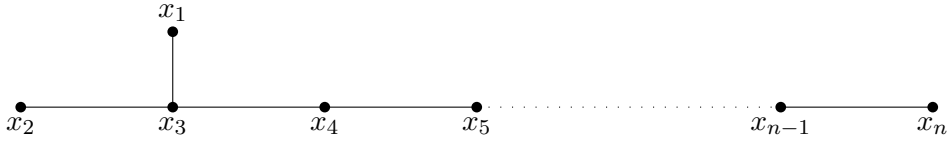
\begin{tikzpicture}
	\draw(-4,0) -- (-4,1);
	\draw (-6,0) -- (0,0);
	\draw[loosely dotted] (0,0) -- (4,0);
	\draw (4,0) -- (6,0);
	\fill (-6,0) circle(2pt)node[anchor=north]{$x_2$};
	\fill (-4,0) circle(2pt)node[anchor=north]{$x_3$};
	\fill (-4,1) circle(2pt)node[anchor=south]{$x_1$};
	\fill (-2,0) circle(2pt)node[anchor=north]{$x_4$};
	\fill (0,0) circle(2pt)node[anchor=north]{$x_5$};
	\fill (4,0) circle(2pt)node[anchor=north]{$x_{n-1}$};
	\fill (6,0) circle(2pt)node[anchor=north]{$x_{n}$};
\end{tikzpicture}

\end{definition}

\section{Representation theory of finite Coxeter groups} 
\subsection{Basics of representation theory}

    \begin{definition}(Representation of a group)
        Let $G$ be a finite group, $V$ a finite dimensional complex vector space. A complex \textbf{representation} of $G$ is a group homomorphism $\phi\colon G\to \gl (V)$. 
    \end{definition}
We denote a representation $\phi: G \rightarrow \mathrm{GL}(V)$ of $G$ by $(\phi, V)$. We will also denote this either by $\phi$ or $V$, whenever the meaning is clear from the context.  The integer $\dim V$ is called the \textbf{degree} of the representation $\phi$.
We now give a few basic example of the representations of a finite group. 
\begin{example}
\label{example-trivial}
     For any group $G$, the map $\phi : G \rightarrow \mathbb C^{\star}$, given by $\phi(g) = 1$ for all $g$. 
\end{example}
This is a one dimension representation of $G$, called the {\bf trivial representation} of $G$. 
\begin{example}
\label{example:cyclic}
    Let $C_n$ be a cyclic group with $n$ elements. Let $\phi : C_n \rightarrow \mathbb C^{\times}$ be a homomorphism, then 
\[ \phi(x)^n = 1 \,\,\mathrm{for\,\, all}\,\, x \in C_n. \] 
Let $\omega$ be a generator of $C_n$ and $\zeta_n$ be a fixed $n^{th}$ primitive root of $1$. 
Then $\phi$ is completely determined by $\phi(\omega)$. Let $\phi_i(\omega) = \zeta_n^i$  for $1 \leq i \leq n$. Thus the $\phi_i$ are distinct one dimensional representations of $C_n$. 
\end{example} 
\begin{example} 
\label{example-three-dimensional-rep-S3}
Let $S_3$ be the symmetric group of degree three.  
Define $\phi : S_3 \rightarrow \mathrm{GL}_3(\mathbb C)$ by 
\[(1) \rightarrow \left[ \begin{matrix} 1 & 0 & 0 \\ 0 & 1 & 0 \\ 0 & 0 & 1 \end{matrix} \right], (12) \rightarrow 
\left[ \begin{matrix} 0 & 1 & 0 \\ 1 & 0 & 0 \\ 0 & 0 & 1 \end{matrix} \right] \]
\[(13) \rightarrow \left[ \begin{matrix} 0 & 0 & 1 \\ 0 & 1 & 0 \\ 1 & 0 & 0 \end{matrix} \right], (23) \rightarrow \left[ \begin{matrix} 1 & 0 & 0 \\ 0 & 0 & 1 \\ 0 & 1 & 0 \end{matrix} \right]\]
\[(123) \rightarrow \left[ \begin{matrix} 0 & 0 & 1 \\ 1 & 0 & 0 \\ 0 & 1 & 0 \end{matrix} \right], (132) \rightarrow \left[ \begin{matrix} 0 & 1 & 0 \\ 0 & 0 & 1 \\ 1 & 0 & 0 \end{matrix} \right]. \] 
\end{example} 
It is easy to check that $\phi$ is a representation of $S_3$ and it has dimension three. The above example is in fact obtained from a general phenomenon by considering the action of a group $G$ on a finite set $X$. We discuss this process of building new representations in the following example:  

\begin{example} 
(Permutation representation) 
\label{example-permutation-representation}Suppose $G$ acts on finite set $X$, that is for each $s \in G$, there is given a permutation $x \mapsto sx$ of $X$ satisfying
\[
1 x = x, s(tx) = (st)x , \quad s,t \in G, x \in X. 
\]

Let $V$ be complex vector space with basis $(e_x)_{x \in X}$. For $s \in G$, let 
\[
\begin{split} \rho: G \rightarrow \mathrm{GL}(V); \\ \rho(s): e_x \mapsto e_{sx}. \end{split} \]
 The $\rho$ is a representation of $G$ with $\mathrm{dim}(\rho) = |X|$.
\end{example} 
The following specific  permutation representation plays a very important role in the representation theory of groups. 
\begin{example}
(Regular Representation) If $V$ is space with basis $(e_g)_{g \in G}$, then above action is called regular representation of $G$.
\end{example} In the area of representation theory of finite groups, we are interested in determining all representations of a finite group $G$. We now discuss some tools that help us in this endeavor.

Let $(\phi, V)$ be a representation of $G$. A subspace $W$ of $V$ is called \(\boldsymbol{G}\)-\textbf{invariant} if $\phi(g) W = W$ for every $g\in G$. In this case, $\phi$ restricts to $W$ and gives rise to a homomorphism $\psi\colon G\to\gl(W)$. Such a representation is called a \textbf{sub-representation} of $(\phi, V)$. 
    \begin{definition}[Irreducibility]
        A representation $\phi\colon G\to\gl(V)$ is called \textbf{irreducible} if it admits no proper sub-representations.  
    \end{definition}
    In other words, $(\phi, V)$ is irreducible if and only if it does not have proper invariant subspaces. Note that any one dimensional representation of $G$ is irreducible. 
    \begin{definition}[Equivalence or isomorphism of representations]
       Two representations $(\phi_{1}, V_{1})$ and $(\phi_{2}, V_{2})$ are said to be equivalent or isomorphic, denoted $(\phi_1, V_1) \cong (\phi_2, V_2)$, if there exists an isomorphism $T : V_{1} \mapsto V_{2}$ such that 
$ \phi_{2}(g)T = T\phi_{1}(g)$ for all $g \in G $. 
    \end{definition}
In matrix notations, if $\dim(V_1) = \dim(V_2) = n$. Then $(\phi_1, V_1) \cong (\phi_2, V_2)$, if there exists $X \in \mathrm{GL}_n(\mathbb C)$ such that \[ X \phi_1(g) X^{-1} = \phi_2(g), \]
for all $g \in G.$
We now give some examples to clarify this notion of equivalence. 
\begin{example}
    Consider $\phi_1 : S_2 \rightarrow \mathrm{GL}_2(\mathbb C)$ by 
\[ (1) \mapsto \left[ \begin{smallmatrix} 1 & 0 \\ 0 & 1 \end{smallmatrix} \right], (12) \mapsto 
\left[ \begin{smallmatrix} 0 & 1 \\ 1 & 0 \end{smallmatrix} \right]. \] 
and $\phi_2 : S_2 \rightarrow \mathrm{GL}_2(\mathbb C)$ by 
\[ (1) \mapsto \left[ \begin{smallmatrix} 1 & 0 \\ 0 & 1 \end{smallmatrix} \right], (12) \mapsto \left[ \begin{smallmatrix} 1 & 0 \\ 0 & -1 \end{smallmatrix} \right]. \] 
Then $\phi_1 \cong \phi_2$. Indeed  $X = \left[ \begin{smallmatrix} 1 & 1 \\ 1 & -1 \end{smallmatrix} \right]$ satisfies 
\[ X \left[ \begin{smallmatrix} 0 & 1 \\ 1 & 0 \end{smallmatrix} \right] X^{-1}  = \left[ \begin{smallmatrix} 1 & 0 \\ 0 & -1 \end{smallmatrix} \right].\]
\end{example}
\begin{example}
    Consider one dimensional representations $\phi_i$ and $\phi_j$ of cyclic group $C_n$ given by 
\[ \phi_i(\omega) = \zeta_n^i \,, \quad \phi_j(\omega) = \zeta_n^j \]
where $\omega$ is generator of $C_n$.We claim that for $i \neq j$, $\phi_i \ncong \phi_j$. We note that if $f : \mathbb C \rightarrow \mathbb C$ is an isomorphism then $f(x) = \lambda x$ for some $\lambda \in \mathbb C^{\star}$. By definition, $\phi_i \cong \phi_j$ if and only if the following diagram commutes. 
\[ 
\xymatrix{
\mathbb C \ar[d]^{\lambda} \ar[r]^{\zeta_n^i} & \mathbb C \ar[d]^{\lambda} \\
\mathbb C \ar[r]^{\zeta_n^j} & \mathbb C }
\]
Therefore, we must have $\lambda \zeta_n^i = \zeta_n^j \lambda$. This is not possible for $i \neq j$. 

\end{example}
\begin{definition}($G$-linear map) Let $(\phi_1, V_1)$ and $(\phi_2, V_2)$ be two representations of finite group $G$. Then a map $T: V_1 \rightarrow V_2$ is called $G$-linear if
\begin{enumerate}
\item $T$ is $\mathbb C$-linear. 
\item $T \circ \phi_1(g) = \phi_2(g) \circ T$ for all $g \in G$. 
\end{enumerate}
\end{definition}
\begin{lemma} The kernel and image of $G$-linear map are $G$-invariant subspaces. 
\end{lemma}

\begin{proof}
 Let $(\phi_1, V_1)$ and $(\phi_2, V_2)$ be representations of $G$. Let $T: V_1 \rightarrow V_2$ be a $G$-linear map.
 Let $W_1 \subseteq V_1$ be the kernel of $T$. We prove that $\phi_1(g)v \in W_1$ for all $g \in G$ and $v \in W_1$. We note that $T (\phi_1(g) v) = \phi_2(g) (T(v)) = 0 $. Therefore $\phi_1(g) v \in W_1 $ for all $g\in G$ and $v \in W_1$. A similar argument proves that the image of $T$ is a sub-representation of $(\phi_2, V_2)$. 
 \end{proof} 
 Another useful tool to construct new representations from the known ones is of the direct sum of representations defined below: 
\begin{definition} (Direct Sum of representations) If $(\phi , V)$ and $(\psi, W)$ be two representations of group $G$, then $(\phi \oplus \psi, V \oplus W)$ given by 
\[ [(\phi \oplus \psi)(g)](v,w) = (\phi(g)v, \psi(g)w) \] 
is a representation of $G$, called the direct sum of $\phi$ and $\psi.$ 
\end{definition}

Note that $\mathrm{dim}(\phi \oplus \psi) = \mathrm{dim}(\phi) + \mathrm{dim}(\psi)$.  
 In terms of matrices, the matrix of $(\phi \oplus \psi)(g)$ is given by  
\[
(\phi \oplus \psi)(g) = \left[ \begin{smallmatrix} \phi(g) & 0 \\ 0 & \psi(g) \end{smallmatrix} \right] 
\]
for all $g \in G.$

For the readers, who are familiar with the language of modules over rings and algebras, representation of a finite group is analogous to studying the modules of certain finite dimensional algebras, called Group algebras of $G$. We briefly discuss this here for the reader's convenience.  

{\bf Group Algebra}
Let $G$ be a finite group and $F$ a field.  
The \emph{group algebra} $F[G]$ is the $F$-vector space with basis $\{g \mid g \in G\}$, i.e.
\[
F[G] = \left\{ \sum_{g \in G} a_g g \;\middle|\; a_g \in F \right\}.
\]
Multiplication is defined by extending the group operation linearly:
\[
\Big( \sum a_g g \Big)\Big( \sum b_h h \Big) = \sum_{g,h} a_g b_h (gh).
\]
Thus $F[G]$ is both a vector space (of dimension $|G|$) and an algebra. 

We note that a vector space $V$ is a complex representation of $G$ is the same as giving $V$ the structure of a \emph{left $\mathbb C[G]$-module} via
\[
\Big( \sum a_g g \Big) \cdot v = \sum a_g \rho(g)(v).
\]
Hence studying representations of $G$ is equivalent to studying $\mathbb C[G]$-modules. Hence representations of $G$ can also be studied using this language. 
The following result relates the notion of irreducible representations and the direct sum of representations.
\begin{proposition} Let $(\phi, V)$ be a complex representation of finite group $G$. The following are equivalent:
\begin{enumerate}
\item $(\phi, V)$ is irreducible. 
\item $(\phi, V)$ can not be written as direct sum of two proper sub-representations.
\end{enumerate}
 
\end{proposition} 
\begin{proof} Here $(1) \implies (2)$ is straight forward. For the other side implication, we show that there is a complimentary invariant subspace $W'$ of a proper subspace $W$ of $V$ such that $
V = W \oplus W'.$
 Let $U$ be an arbitrary complement of $W$ in $V$ and let 
\[
\pi_0 : V \rightarrow W
\]
be the projection given by the direct sum decomposition $V = W \oplus U$. We average the map $\pi_0$ over $G$, that is define an onto map $\pi : V \rightarrow W$ by,
\[
\pi(v) = \frac{1}{|G|} \sum_{g \in G} \phi(g) (\pi_0 (\phi(g)^{-1} v)).
\]
Then $\pi$ is a $G$-linear. Its kernel, say $W'$, is the required $G$-invariant complement of $W$.  
\end{proof} 
From above, we directly obtain the following result:  
\begin{theorem}(Maschke) Every complex representation of finite group $G$ is a direct sum of irreducible representations of $G$. 
\end{theorem}
This reduces our question of determining all representations of $G$ to the following:
\begin{question} What are all the finite-dimensional inequivalent irreducible complex representations of a given finite group $G$? 
\end{question}
We use $\mathrm{Irr}(G)$ to denote a complete set of in-equivalent irreducible representations of $G$. 
Let us answer this question to determine $\mathrm{Irr}(C_n)$ for the $C_n$ of order $n$. Let $\phi : C_n \rightarrow \mathrm{GL}_m(\mathbb C)$ be a  homomorphism. 
It is easy to observe that every finite order complex matrix is diagonalizable. 
Therefore, we note that the matrices $\phi(x)$ for $x \in C_n$ are pairwise commuting and diagonalizable. Hence $\phi(x)$ are simultaneously diagonalizable for all $x \in C_n$. 
Therefore $\phi$ is easily seen to be a direct sum of one dimensional representations of $C_n$. This proves that every irreducible representation of $C_n$ is one dimensional. Further, we have determined all of these in Example~\ref{example:cyclic}. In this case, we will have
\[
\phi \cong \oplus_{i} \phi_i^{\oplus m_i},
\]
for $1 \leq i \leq n$. The integer $m_i$ is called the {\bf multiplicity} of $\phi_i$ in $\phi$. In general it is defined as below:
\begin{definition}(Multiplicity) 
Let $V$ be a representation of a group $G$ such that
\[
V \cong \oplus V_i^{m_i}, 
\]
where $V_i$'s are pairwise in-equivalent irreducible  representations of $G$. Then $m_i$ is called the multiplicity of $V_i$ in $V$. It is denoted by $\langle V, V_i \rangle.$ Note that $m_i = \dim_{\mathbb C} \mathrm{Hom}_G(V, V_i).$ 
\end{definition}

The following result is very useful to determine if a representation is irreducible or not and also whether any two irreducible representations are equivalent or not. 
\begin{lemma} (Schur's Lemma) Let $\phi_1 : G \rightarrow \mathrm{GL}(V_1)$ and $\phi_2 : G \rightarrow \mathrm{GL}(V_2)$ be two irreducible representations of $G$, and let $T$ be a linear mapping of $V_1$ into $V_2$ such that $\phi_2(g) \circ T = T \circ \phi_1(g)$ for all $g \in G$. Then :
\begin{enumerate}
\item If $\phi_1$ and $\phi_2$ are not isomorphic, we have $T = 0$.
\item If $V_1 = V_2$ and $\phi_1 = \phi_2$, then $T(x) = \lambda x$ for $x \in V$, for some scalar $\lambda \in \mathbb C$. 
\end{enumerate} 
\end{lemma}

 \begin{proof}
  Suppose $T: V_1 \rightarrow V_2$ such that $T \neq 0$. The $G$-linearity of $T$ implies that both kernel and image of $T$ are $G$-invariant subspaces. The irreducibility of $\phi_1$ and $\phi_2$ implies $T$ is an isomorphism. 

We now assume that $(\phi_1, V_1) = (\phi_2, V_2)$.  Let $\lambda$ be an eigenvalue of $T$ (recall field is $\mathbb C$). 
Let $T' = T -\lambda$. Then $Ker(T') \neq 0$. 
 We note that $T'$ is also a $G$-linear map. 
 By (1), we must have $T' = 0$. That is $T(x) = \lambda(x)$ for all $x \in V_1$. 
 \end{proof}
 
    \begin{definition}[Character]
        Let $\rho\colon G\to\gl(V)$ be a finite dimensional representation. The \textbf{character} $\chi_\rho$ of $\rho$ is the complex valued function $\chi_\rho\colon G\to\mathbb C$ defined by $\chi_\rho(g)=\tr(\rho(g))$ for every $g\in G$. 
    \end{definition}
If $\rho$ is an irreducible representation of $G$, then its character $\chi_\rho$ is called an irreducible character of $G$. It is easy to see that $\chi_{\rho} = \chi_{\rho'}$ for $\rho \cong \rho'.$ It turns out that converse is also true. As conjugate matrices have the same trace, it follows that $\chi_\rho$ is constant on a conjugacy class of $G$. Such a function is called a \textbf{class function} of $G$. 
 Let $\mathbb C^{\mathrm{class}}(G)$ be the space of all class functions of $G$. For $f, f' \in \mathbb C^{\mathrm{class}}(G),$ define
 \[
 \langle f, f' \rangle = \frac{1}{|G|}\sum_{g \in G} f(g) \overline{f'(g)}.
 \]
Then it is easy to see that $\mathbb C^{\mathrm{class}}(G)$ is an inner product space with respect to the inner product. The following results, whose proof we skip, are highly beneficial result to determine all irreducible representations of a finite group. For proofs, see~\cite{MR450380}. 
 \begin{proposition}
\label{prop:identify-irreucible}   \begin{enumerate}
         \item The set $\{\chi_\rho \mid \rho \in \mathrm{Irr}(G) \}$  is an orthonormal basis of $\mathbb{C}^{\mathrm{class}}(G).$ 
         \item A representation $\rho$ is irreducible if and only if $\langle \chi_\rho, \chi_\rho \rangle = 1.$ 
         \item For representations $\rho$, $\rho'$ (not necessarily irreducible) of $G$, we have $\chi_{\rho} = \chi_{\rho'}$ if and only if $\rho \cong \rho'.$
     \end{enumerate}
 \end{proposition}

 \begin{theorem}
 \label{thm:number-irr-conj} The number of inequivalent irreducible representations of a finite group is equal to the number of its conjugacy classes.
\end{theorem}

\begin{theorem}
\label{thm:dimension-square}If $\{\rho_1, \rho_2, \ldots, \rho_t\}$ is a complete set of all in-equivalent irreducible representations of group $G$, then
\[
|G| = \sum_{i=1}^t \mathrm{dim}(\rho_{i})^2. 
\] 
\end{theorem}
Let us use the above results to describe $\mathrm{Irr}(S_3)$. The group $S_3$ has three conjugacy classes and the only way to write $6$ as the sum of three squares is:
\[
6 = 1 + 1 + 2^2.
\]
Hence $S_3$ must have two one dimensional representations and one two dimensional irreducible representation. Trivial representations, defined in \Cref{example-trivial}, gives a representation of $S_3$ of dimension one. Another one dimensional representation of $S_3$ is obtained by considering the $\mathrm{sign}$ function from $S_3$ to $\{ \pm 1\}.$ To obtain the remaining two dimensional representation of $S_3$, we consider the representation $\phi$ discussed in \Cref{example-three-dimensional-rep-S3}. Observe that $\phi((12))$ and $\phi((23))$ do not commute, hence $\phi$ cannot be a direct sum of one dimensional representations of $S_3$. Hence $\phi = \phi_1 \oplus \phi_2,$ for some one-dimensional representation $\phi_1$ and a two-dimensional representation $\phi_2.$ Hence we obtain our remaining two-dimensional irreducible representation of $S_3$. 

Infact, let us  find this representation explicitly. Notice that the representation $\phi$ is obtained by considering the permutation action of $S_3$ on $\{ 1, 2, 3\}$ by $\sigma(e_x) = e_{\sigma x}$ as discussed in \Cref{example-permutation-representation}. Let us denote this representation by $(\phi, V)$. Consider $W = \mathrm{Span}_{\mathbb C}(e_1 + e_2 + e_3)$. Then $W$ is an invariant subspace of $V$. The quotient space $W' = V/W$ is also a representation of dimension $2$ with specific action given by 
\[
\phi_2: G \rightarrow \mathrm{GL}(W'); \,\, \phi_2((12)) = \begin{pmatrix} 0 & 1 \\ 1 & 0 \end{pmatrix},\,\,  \phi_2((23)) =  \begin{pmatrix} -1 & 0 \\ -1 & 1 \end{pmatrix}.
\]
Since $\phi_2$ has dimension two and $\phi_2((12))$ and $\phi_2((23))$ do not commute, so $\phi_2$ must be the two dimensional irreducible representation of $S_3$. We show the irreducibility of $\phi_2$ also by evaluating its character value. We note that $\chi_{\phi_2}$ is given by $\chi_{\phi_2}(e) = 2$, $\chi_{\phi_2}(12) = \chi_{\phi_2}(23) = \chi_{\phi_2}(13) = 0,$ and $\chi_{\phi_2}(123) = -1, \chi_{\phi_2}(132) = -1.$ 
Therefore,
\[
\langle \chi_{\phi_2}, \chi_{\phi_2} \rangle = \frac{1}{6}(4 + 0 + 2(-1)(-1)) = 1. 
\]
By \Cref{prop:identify-irreucible}, the representation $\phi_2$ is an irreducible representation of $S_3.$  

\begin{comment}
{\bf Goal:} To determine irreducible representations of infinite Coxeter groups.

\[
I_2(m)=\langle r,s \mid r^m=1, s^2=1, srs=r^{-1} \rangle
\]
$H =\langle r \rangle \subseteq I
_2(m)$. Then $H \unlhd I_2(m)$ such that $[G:H]=2$.
\\
Any irreducible representation of $H$ is one dimensional and these are exactly $m$ many in number.
\end{comment}

We will now focus on the representations of infinite families of Coxeter groups. Recall these groups are given as below:  
\begin{itemize}
\item $A_n$: These are the symmetric groups with $A_n = S_{n+1}.$ 
\item $B_n$ These are hyper-octahedral groups given by $B_n =(\Z/2\Z)^n \rtimes S_n$.
\item $D_n=$: A specific index two subgroup of $B_n$, usually called as even-signed permutation group on 
$n$ letters.
\item $I_2(m)$: The Dihedral groups of order $2m$. 
\end{itemize}

\subsection{A construction of representations of $S_n$}
  
Consider the symmetric group on $n$ letters $S_n$. 
\begin{definition}
[Partition of a Positive Integer]
A \textbf{partition} of a positive integer $n$ is a finite sequence of positive integers
\[
\lambda = (\lambda_1, \lambda_2, \dots, \lambda_k)
\]
such that:
\begin{enumerate}
  \item $\lambda_1 \geq \lambda_2 \geq \cdots \geq \lambda_k > 0$, and
  \item $\lambda_1 + \lambda_2 + \cdots + \lambda_k = n$.
\end{enumerate}

\end{definition}
We will use $\lambda \vdash n$ to denote that $\lambda$ is a partition of $n$.  For any partition $\lambda=(\lambda_1,\lambda_2,..,\lambda_k) \vdash n$, all permutations in $S_n$ of cycle type $\lambda$ form a conjugacy class. The following is well known regarding the Symmetric group $S_n$. 

\begin{theorem}
    The set of all conjugacy classes of $S_n$ is in natural bijection with the set of all partitions of $n$.
\end{theorem}
Motivated by \Cref{thm:number-irr-conj}, we will construct in-equivalent irreducible representations of $S_n$ parametrized by the partitions of $n$. 

\begin{definition}[Young Diagrams]
    Given a partition $\lambda=(\lambda_1,\lambda_2,..,\lambda_k) \vdash n$, a Young diagram $T_{\lambda}$ of shape $\lambda$ is left and top alligned shape of empty boxes such that $i^{th}$ row has $\lambda_i$ boxes.
\end{definition} 

\noindent{\bf Examples : } 
\\

\begin{center}
\begin{tabular}{c@{\hskip 2cm}c}
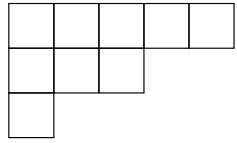
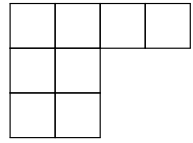
\ydiagram{5,3,1} & \ydiagram{4,2,2} \\
 $T_{(5,3,1)}$ & $T_{(4,2,2)}.$
\end{tabular}
\end{center}
\begin{definition}[Young Tableaux] Given a partition $\lambda=(\lambda_1,\lambda_2,..,\lambda_k) \vdash n$, a filling of $T_\lambda$ with numbers from $1$ to $n$ without repetition is called a young tableaux of shape $\lambda$. 
    
\end{definition}

\noindent{\bf Examples : } \begin{center}
\begin{tabular}{c@{\hskip 2cm}c@{\hskip 2cm}c}
\begin{ytableau}
1 & 2 & 4 & 7 & 9 \\
3 & 5 & 8 \\
6
\end{ytableau}
&
\begin{ytableau}
1 & 3 & 5 & 7 & 10 \\
2 & 4 & 8 \\
6
\end{ytableau}
&
\begin{ytableau}
1 & 3 & 3 & 7 & 9 \\
4 & 5 & 8 \\
6
\end{ytableau}
\\[10pt]
Young Tableau & Young Tableau & \textbf{Not a Young Tableau }
\end{tabular}
\end{center}

It is easy to observe that the number of the young tableaux to shape  $\lambda$ for  $\lambda \vdash n$ is $n!$. 
For $\lambda \vdash n, \sigma \in S_n $ the young tableau $T_\lambda(\sigma)$ of shape $\lambda$ associated to $\sigma \in S_n$ is the diagram $T_\lambda$ filled with entries $\sigma(1) \dots \sigma(n)$ moving in rows each from left to right and down the columns from up to down. 
\bigskip

\noindent{\bf Examples : } Let id be the identity permutation and $\sigma=(1 5 8 6)(2 3 4 7)$ in $S_9$. Then
\begin{center}
\begin{tabular}{c@{\hskip 2cm}c}
$T_{(5,3,1)}(\mathrm{id}) =$ 
\ytableausetup{boxsize=normal}
\begin{ytableau}
1 & 2 & 3 & 4 & 5 \\
6 & 7 & 8 \\
9
\end{ytableau}
&
$T_{(5,3,1)}(\sigma) =$
\begin{ytableau}
5 & 3 & 4 & 7 & 8 \\
1 & 2 & 6 \\
9
\end{ytableau}
\end{tabular}
\end{center}
The group $S_n$ acts on the set of young tableaux of shape $\lambda \vdash n$ via 
\[
\tau(T_\lambda(\sigma))=T_\lambda(\tau \sigma). 
\]
\begin{definition}(Row(Column) Group of $T_{\lambda}(\sigma)$) Given $T_\lambda(\sigma)$, its row group $R_\lambda(\sigma)$ (column group $C_\lambda(\sigma)$) is the set of partitions of $S_n$ that may permute the entries of each row(column) of $T_\sigma$ but do not allow cross row (column) permutation of the entries of $T_\lambda(\sigma)$.  
\[
R_\lambda(\theta)=\{\sigma \in S_n \mid \sigma \text{ stabilizes the rows of } T_\lambda(\theta) \text{ ( not necessarily entrywise)}\}
\]
\end{definition}
\noindent{\bf Examples : }\begin{enumerate}
    \item For
$T_{(5,3,1)}(\text{id})$, we have  $$R_{(5,3,1)}(\text{id})=S_{\{1,2,3,4,5\}}\times S_{\{6,7,8\}}, \,\,  C_{(5, 3,1)}(\text{id}) = S_{1, 6, 9} \times S_{2,7} \times S_{3, 8}.$$
\item For $\sigma=(1586)(2347)\in S_9$ and $T_{(5,3,1)}(\sigma)$, we have $$R_{(5,3,1)}(\sigma)=S_{\{3,4,5,7,8\}}\times S_{\{1,2,6\}}, \,\, C_{(5,3,1)}(\sigma)= S_{1, 5, 9} \times S_{2, 3} \times S_{4,6}.$$ 
\end{enumerate}

For $\lambda \vdash n,$ let $R_\lambda$ and $C_\lambda$ denote the row group and column group of $T_\lambda(\text{id})$.  Let 
\[
a_\lambda= \underset{\sigma\in R_\lambda}{\Sigma} \sigma \quad \text{and} \quad b_\lambda= \underset{\tau\in C_\lambda}{\Sigma}  sgn(\tau) \tau  \quad
\]
 Let $c_\lambda=a_\lambda b_\lambda$. Define $V_\lambda= \mathbb{C}[S_n]c_\lambda$.
 \bigskip
\begin{theorem}
\label{thm:Sn-construction-1}    
    Let $\lambda \vdash n$ and $V_\lambda$ as defined above. Then
    \begin{enumerate}
        \item For $\lambda \vdash n$, $V_\lambda$ is an irreducible representation of $S_n$. 
        \item $V_\lambda \not \cong V_\mu$ for $\lambda \neq \mu$.
        \item Any irreducible representation of $S_n$ is isomorphic to $V_\lambda$ for some $\lambda \vdash n$.
    \end{enumerate}
\end{theorem}
For a proof of this result, see. We now illustrate the above theorem for the example of $S_3$. 
\begin{example}(Example of $S_3$) 
For $n = 3,$ we have three partitions:
\[
 (3), \,\, (2,1), \,\, (1,1,1). 
\]
We now explain representations for each of these partitions. 

\noindent {\bf $\lambda = (3)$:} We have $R_{(3)} = S_3$ and $C_{(3)} = (e).$ Therefore 
\[
c_{(3)} = \sum_{\sigma \in S_3}\sigma \in \mathbb C[S_3]. 
\]
for this case. We note that $\tau c_{(3)} = c_{(3)}.$ Hence $V_{(3)}$ is one dimensional representation of $S_3$ generated by $c_{(3)}.$ 

\noindent {\bf $\lambda = (1,1,1)$:} For this case, $R_{(1,1,1)} = (e)$ and $C_{(1,1,1)} = S_3.$ Hence
\[
c_{(1,1,1)} = \sum_{\sigma \in S_3} \mathrm{sgn}(\sigma) \sigma. 
\]
We note that $\tau c_{(1,1,1)} = \mathrm{sgn}(\tau) c_{(1,1,1)}.$ Hence $V_{(1,1,1)}$ is the non-trivial one dimensional  representation of $S_3$.

\noindent{ $\lambda = (2,1)$:} For this case $R_{(2,1)} = S_{1,2}$ and $C_{(2,1)} = S_{1,3}$. Hence we have $a_{(2,1)} = e + (12)$ and $b_{(2,1)} = e-(13).$ Hence
\[
c_{(2,1)} = e + (12) - (13) - (12)(13).   
\]
We note that $(12)c_{(12)} = c_{(12)}$ and $(13)c_{(12)} = --e + (13) + (13)(12) - (23).$ Since $S_3$ is generated by $(12)$ and $(13)$, we obtain that $V_{(21)}$ is a two dimensional representation space. Let $\rho: S_3 \rightarrow \mathrm{GL}(V_{(2,1)}$ be the corresponding representation of $S_3$ on $V_{(2,1)}$.
The matrices of $(12)$ and $(13)$ with respect to the basis $\{ c_{(2,1)}, (13)c_{(2,1)}\}$ are given by:
$\rho((12)) = \begin{bmatrix}
    0 & -1 \\ 1 & -1
\end{bmatrix}$
and $\rho((13)) = \begin{bmatrix}
    0 & 1 \\ 1 & 0
\end{bmatrix}.$ By considering the character values of this representation, we observe that it agrees with the already constructed character of $S_3$ in discussion following \Cref{thm:dimension-square}. 
\end{example}
This example illustrates that although \Cref{thm:Sn-construction-1}, gives a construction of representations of $S_n$, it is still difficult to work with it. For this reason, we find many different constructions of representations of $S_n$ to cater different problems in hand. For completeness, we include below the expression of the dimension of $V_{\lambda}$'s. See~\cite{MR1824028} for a proof.  
\begin{definition} 
 {\bf Hook number of a box of  a young diagram: } For $\lambda \vdash n$, consider the young diagram $T_\lambda$. The hook number of any box of $T_\lambda$ is the number of boxes exactly on the right and exactly below the box including the box itself. 
  \end{definition} 
  In the next example, we write the hook numbers of each box in the box itself. 
 
   \begin{center} 
  \begin{tabular}{c@{\hskip 2cm}c}
  \begin{ytableau}
        7&6&4&2&1\\
        4&3&1\\
        2&1
    \end{ytableau}
  \end{tabular}
  \end{center}  
 \bigskip
  Let $h_\lambda$  be the products of all hook numbers of $T_\lambda$. In the above example $h_\lambda = 7.6.4.4.3.2.2.1.1$ Then 
   \[
   \text{dim} (V_\lambda)=\frac{n !}{h_\lambda}.
   \]
   This is called the {\bf hook length formula}.

\subsection{Representation theory of Hyperoctahedral groups} 
In this section, we discuss the representation theory of finite Coxeter groups of type $B_n$. These groups are also called the Hyperoctahedral groups. Recall that for an integer $n\geq 2$, the hyperoctahedral group $B_n$ is defined as the semi-direct product $(\mathbb{Z}/2\mathbb{Z})^n \rtimes S_n$, where the action of $S_n$ on $(\mathbb{Z}/2\mathbb{Z})^n $ is given as follows:
\[
\sigma . (a_1, \cdots, a_n)=(a_{\sigma(1)}, \cdots, a_{\sigma(n)}) \quad \forall \quad \sigma \in S_n , a_i \in \{\pm 1\}
\]
Since $(\mathbb{Z}/2\mathbb{Z})^n$ is a normal abelian subgroup of $B_n$, we can apply Wigner-Mackey's little group method for constructing the irreducible representations of $B_n$.

To describe this method, we need two more important concepts from the representations of finite groups:  {\it induced representation} and {\it restriction}. We define these below. 
\begin{defn}(Induced representation)
Let $H$ be a subgroup of a finite group G, and let $(\psi, U)$ be a representation
of $H$. 
Let $$V = \{f : G \rightarrow U | f(hg) = \psi(h) f(g), h \in H, g \in G \}.$$ Then $G$ acts on
$V$ by right translation, that is, $\phi : G \rightarrow \mathrm{Aut}(V)$ is defined by
\[
[\phi(g)(f )] (g') = f (g'g) \quad g, g' \in G, f \in V. 
\]
It follows that $(\phi, V)$ is a representation of $G$. This is called the induced representation of $(\psi, U)$ from $H$ to $G$, denoted by $\mathrm{Ind}_H^G(\psi)$.  
\end{defn}
 By definition, we have $$\mathrm{dim}(\mathrm{Ind}^G_{H}(\psi)) = |G/H|.\mathrm{dim}(\psi).$$ 
 For a representation $(\phi, V)$, restricting the action of $\phi$ on a subgroup $H$ of $G$, we obtain a representation of $H$, denoted by $\mathrm{Res}^G_H(\phi)$. This representation is called the restriction of $(\phi, V)$ from $G$ to $H$. The following result related these two notions of the representations. 
 \begin{theorem}(Frobenius reciprocity)
 \label{thm-frobenius-reciprocity}
Let $\phi$ be a representation of group $G$ and $\rho$ be a representation of $H$. Then
\[
\langle \mathrm{Ind}_H^G(\rho), \phi  \rangle_G = \langle \rho, \mathrm{Res}_H^G(\phi) \rangle_H.    
\]
 \end{theorem}

\noindent {\bf Wigner-Mackey little group method:} Let $G= A \rtimes H $ be a finite group, such that $A $ is an abelian normal subgroup of $G$. Let $\mathrm{Irr}(A)$ denote the set of equivalence classes of irreducible complex representations of $A$. Since $A$ is abelian, $\mathrm{Irr}(A) \cong A$ and every $\phi \in \mathrm{Irr}(A)$ is one dimensional. 

The group $H$ acts on $\mathrm{Irr}(A)$ via
\[
(h \cdot \phi) (a)=\phi^h(a) := \phi(h^{-1}ah) \quad \text{ for all } a\in A, h \in H.
\]
Let $ \{ \phi_1,\phi_2,\cdots, \phi_k \} $ denote the orbit representatives of this action. Define the subgroup $H_i$ of $H$ as follows
\[
\mathrm{Stab}_{H}(\phi_i)=H_i := \{ h \in H \mid \phi_i^h=\phi_i \}.
\]
Then there exists an extension $\tilde{\phi_i}$ of $\phi_i$ to $A \rtimes H_i$ by
\[
\tilde{\phi_i}(ah)=\phi_i(a) \quad \text{for\,\,  all}\,\,  a\in A, h\in H.
\]
For any $H_i$ and $(\psi , V)\in \mathrm{Irr}(H_i)$, the map $\tilde{\phi_i} \otimes \psi \colon A \rtimes H_i \longrightarrow GL(V)$  defined by 
\[
 (\tilde{\phi_i} \otimes \psi) (a,h) := \tilde{\phi_i}(a,h)\psi(h)
\]
is an irreducible representation of $A \rtimes H_i$. This helps us to describe all irreducible representations of $A \rtimes H$.  
\begin{theorem}
\label{thm:mackey-wigner-little-group}
Let $\{\phi_1, \cdots, \phi_k \}$ denote the orbit representatives of $\mathrm{Irr}(A)/H$ and $H_i$ be the stabilizers of $\phi_i$ in $H$. 
\begin{enumerate} 
\item Any irreducible representation of $G=A \rtimes H$ is of the form $\mathrm{Ind}_{A \rtimes H_i}^{G}(\tilde{\phi_i} \otimes \psi)$ for $1 \leq i \leq k$ and $\psi \in \mathrm{Irr}(H_i)$.
\item $\mathrm{Ind}_{A \rtimes H_i}^{G}(\tilde{\phi_i} \otimes \psi) \cong \mathrm{Ind}_{A \rtimes H_j}^{G}(\tilde{\phi_j} \otimes \psi')$ if and only if $i=j$ and $\psi\cong \psi'$. 
\end{enumerate}
\end{theorem}
The following is an immediate corollary of the above result. 
\begin{corollary}
\label{cor:Bn correspondance}
    \[
    \mathrm{Irr}(G) \longleftrightarrow \sqcup_{i=1}^k \{ \psi \mid \psi \in \mathrm{Irr}(H_i)\}.
    \]
\end{corollary}
We use this method to give a construction of the irreducible representations of $B_n$.

\subsection{Characters of $(\mathbb{Z}/2\mathbb{Z})^n$ and their stabilizers in $S_n$}
We first describe the characters of the normal abelian subgroup $(\mathbb{Z}/2\mathbb{Z})^n$ of $B_n$.  

Let $\mathbb{Z}/2\mathbb{Z} = \{ \pm 1\}.$ For $ x \in  \mathbb{Z}/2\mathbb{Z}$ and  $a \in \{0, 1\}$, define a character $\psi_{a}$ of $\mathbb{Z}/2\mathbb{Z}$ by
\[
\psi_{a}(x)=(x)^{a}.
\]
For $(a_1,a_2,\dots,a_n) \mid a_i\in \{0,1 \}$, we have $\psi_{a_1}\times \psi_{a_2} \times \dots\times \psi_{a_n} $ is a character of $(\mathbb{Z}/2\mathbb{Z})^n$.
In fact we have
\[
 \{ (a_1,a_2,\dots,a_n) \mid a_i\in \{\pm{1} \} \} \cong \mathrm{Irr}((\mathbb{Z}/2\mathbb{Z})^n), 
\]
where bijection is obtained by mapping $(a_1,a_2,\dots,a_n)$ to $\psi_{a_1}\times \psi_{a_2} \times \dots\times \psi_{a_n} $. 
Any character of $(\mathbb{Z}/2\mathbb{Z})^n$ is of the form $\psi_{(a_1, a_2, \cdots, a_n)}$, where $\psi_{(a_1, a_2, \cdots, a_n)} =\psi_{a_1}\times \psi_{a_2} \times \dots\times \psi_{a_n} $ for some $a_i's$. 
 For $\sigma \in S_n,$
 we have 
 \begin{eqnarray} 
 \psi_{(a_1, a_2, \cdots, a_n)}^\sigma ((x_1, x_2, \cdots, x_n)) & = &  (\psi_{a_1}\times \psi_{a_2} \times \dots\times \psi_{a_n}) ((x_{\sigma(1)}, x_{\sigma(2)}, \cdots, x_{\sigma(n})) \nonumber \\ 
 & = & \prod_i \psi_{a_i}(x_{\sigma(a_i)}) \nonumber \\
 & = & \prod_i \psi_{\sigma^{-1}(a_i)}(x_{i}). 
 \end{eqnarray}

Hence $S_n$ acts on $ \psi_{(a_1, a_2, \cdots, a_n)}$ by permuting $a_i's$. Denote by $\phi_i$ an orbit representatives of this action. Then
\[
\mathrm{Stab}_{S_n}\phi_i=S_i \times S_{n-i}.
\]
By \Cref{cor:Bn correspondance}, we have 
\[
\mathrm{Irr}(B_n) \longleftrightarrow \cup_{i=0}^n \{(\phi,\psi) \mid \phi \in \mathrm{Irr}(S_i),\,\, \psi \in \mathrm{Irr}(S_{n-i})\}.
\]
\begin{definition}[Complementary Partitions] Two partitions $\lambda \vdash a$ and $\mu \vdash b$ with $a,b \geq 0$ are said to be complementary partitions of an integer $n$ if $a+b=n$. 
    
\end{definition}
An ordered pair $(\lambda,\mu)$ of complementary partitions of $n$ is denoted by $(\lambda,\mu) \models n$. For $(\lambda,\mu)\models n$ such that $\lambda \vdash a$ and $\mu \vdash b$, we have a representation $V_{\lambda}\otimes V_{\mu}$ of the group $S_a \times S_b$. Let $U_{(\lambda,\mu)} = \mathrm{Ind}_{S_a \times S_b} (V_\lambda \times V_\mu)$. By \Cref{thm:mackey-wigner-little-group}, we have the following: 
\begin{theorem}
\label{thm:irr-Bn}
    The set $\{ U_{\lambda,\mu} \mid (\lambda,\mu) \models n\}$ is a complete set of mutually in-equivalent irreducible representations of $B_n$.
\end{theorem}
By the dimension formulae for the induced representation, we have $$\mathrm{dim}(U_{(\lambda,\mu)}) = \frac{|S_n|}{|S_a||S_b|} \text{dim}(V_\lambda) \text{dim}(V_\lambda).$$
By \Cref{thm:number-irr-conj} and \Cref{thm:irr-Bn} , we obtain the following: 
\begin{theorem}
    The set of conjugacy classes of $B_n$ is naturally bijective with the set of pairs of complementary partitions of $S_n$.
\end{theorem}
It is an interesting exercise to find this bijection directly without invoking the representation theory of $B_n$. 
\begin{example}[Example of $B_3$]
Recall $B_3 = (\mathbb Z/2\mathbb Z \oplus \mathbb Z/2\mathbb Z \oplus \mathbb Z/2\mathbb Z) \rtimes S_2.$ The characters of $\mathbb Z/2\mathbb Z \oplus \mathbb Z/2\mathbb Z \oplus \mathbb Z/2\mathbb Z$ under the action of $S_3$ have four orbits, namely $\psi_{(0,0,0)},$ $\psi_{(0,0,1)}$, $\psi_{(0,1,1)}$,  $\psi_{(1,1,1)}.$ We now discuss the contribution of each of these to the representations of $B_3$. 
\begin{itemize}
    \item For $\psi_{(0,0,0)}$ and $\psi_{(1,1,1)}$: In both of these case, the stabilizer is $S_3$. The group $S_3$ has two representations of dimension one and one representation of dimension two. Hence from these two orbits, we obtain four representations of dimension one and two representations of dimension two.  \item For $\psi_{(0,0,1)}$ and $\psi_{(0,1,1)}:$ 
    In this case, the stabilizer is isomorphic to the group $S_2.$ Since $S_2$ has two representations of dimension one. Hence, we obtain four representations of dimension three from these two cases. 
\end{itemize}

\end{example}

\subsection{A construction of representations of $D_n$}
In this section, we briefly mention about the construction of irreducible representations of $D_n$.    
Recall that, $D_n = K \rtimes S_n$, where $K$ is the Kernel of the map: 
$$\mathrm{det}: \{\left(\begin{smallmatrix}
    a_1 & 0 & \dots & 0\\
    0 & a_2 & \dots & 0\\
    \dots & \dots & \dots & \dots\\
    0 & 0 & \dots & a_n
\end{smallmatrix}\right) \}  \rightarrow \{ \pm 1\}.$$
We note that $D_n$ is an index two, hence normal, subgroup of $B_n$. Therefore, the representation theory of $D_n$ is closely related to the representation theory of $B_n$ due to the following result.

\begin{theorem}
    Let $G$ be a group and $N$ be its normal subgroup of index $2$. One of the following holds for any  $\rho \in \mathrm{Irr}(G)$: 
    \begin{enumerate}
        \item[(a)] The representation $\mathrm{Res}_{N}^{G}\rho$ is irreducible. In this case we denote $\mathrm{Res}_{N}^{G}\rho$ by $\rho^0$.
        \item[(b)] The representation $\mathrm{Res}_{N}^{G}\rho$ satisfies $\mathrm{Res}_{N}^{G}\rho \cong \rho^+ \oplus \rho^-$, where $\rho^+$ and $\rho^-$ are irreducible representations of $N$ and $\rho^+ \not\cong \rho^-$.
        \end{enumerate}
         Hence, every irreducible representation of $N$ is equivalent to either $\rho^0$, $\rho^{+}$, or $\rho^-$ for some $\rho \in \mathrm{Irr}(G)$.
   
\end{theorem}
\begin{remark}
    For the above result, it is possible that there exists two representations $\rho$ and $\rho'$ such that $\rho^0 \cong (\rho')^0.$ Hence to determine the complete set, we need to understand these cases as well. 
\end{remark}

Since $D_n \unlhd B_n$ and $[B_n:D_n]=2$, we can use the above result to determine irreducible representations of $D_n$. We obtain the following:
\begin{enumerate}
    \item $\mathrm{Res}_{D_n}^{B_n}U_{(\lambda,\mu)}$ is irreducible if and only if $\lambda \neq \mu$.
    \item $U_{(\lambda,\mu)}^0 \cong  U_{(\lambda',\mu')}^0$ if and only if $\{ \lambda, \mu \} = \{ \lambda', \mu' \}$.
    \item $\mathrm{Res}_{D_n}^{B_n}U_{(\lambda,\lambda)}\cong U_{(\lambda,\lambda)}^+ \oplus U_{(\lambda,\lambda)}^-$. 
    
\end{enumerate}
Hence
\[
\mathrm{Irr}(D_n) \longleftrightarrow \{ U_{(\lambda,\mu)}^0 , U_{(\lambda,\lambda)}^+,  U_{(\lambda,\lambda)}^- \mid (\lambda,\mu)\models n, (\lambda,\lambda)\models n, \lambda \geq \mu\}
\]
and \begin{enumerate}
    \item $\text{dim}(U_{(\lambda,\mu)}^0)=\text{dim}(U_{(\lambda,\mu)})$.
    \item $\text{dim}(U_{(\lambda,\lambda)}^+)=\text{dim}(U_{(\lambda,\lambda)}^-)=\frac{\text{dim}(U_{(\lambda,\lambda)})}{2}$.
\end{enumerate}
\begin{remark}
   The group $I_2(m)$ has a normal cyclic subgroup of index two. Hence the irreducible representations of the dihedral group $I_2(m)$ can be found by the above method and \Cref{thm-frobenius-reciprocity}.
\end{remark}

\section{Further discussion and questions}
  
Many open problems regarding the Coxeter groups are included in \cite{brenti2024openproblemscoxetergroups}. In this section, we list a few additional open problems related to the representations of finite Coxeter groups that may be of an interest to the reader. 
\begin{enumerate}
\item Let $G$ be a finite Coxeter group. For irreducible representations $\chi$ and $\rho$ of $G$, it is an open problem to describe the constituents and their multiplicities of the tensor product $\chi \otimes \rho$. These problems remain open  for groups of type $A_n$ (the symmetric groups); see \cite{MR3219569} and \cite{MR1722888}.  
\item(Saxl conjecture for symmetric groups $S_n$) Let $n \in \mathbb N$ be such that $n = 1 + 2 +3+ 4+ \cdots + k$ for some $k$. The Saxl conjecture states that the tensor square $V_\lambda \otimes V_\lambda$ for $\lambda = (k, k-1, \ldots, 1)$ contains every irreducible representation of $S_n$ as a constituent. This conjecture is still open; see \cite{MR3466368} , \cite{MR3436397} and \cite{MR3357782}. Interestingly,  the tensor cube version of Saxl conjecture, that is $V_\lambda \otimes V_\lambda \otimes V_\lambda$ contains all irreducible representations of $S_n$, has recently been proved by Harman-Ryba~\cite{MR4591596}

 \item The tensor square conjecture for $S_n$ is parallel to the Saxl conjecture for all $n$ and is stated as follows: 
\begin{conjecture}(Tensor-square Conjecture)
For each $n$, there exists a partition $\lambda$ of $n$ such that $V_\lambda \otimes V_\lambda$  contains all irreducible representations of $S_n$. \end{conjecture} 
 
\item Generalizations of Saxl conjecture and Tensor square conjecture to other finite Coxeter groups beyond type 
$A_n$
 are under investigation; see~\cite{chen2024spinrepresentationsfinitecoxeter}.

 \item Classification of projective representations of finite Coxeter groups remains incomplete, particularly in exceptional and non-crystallographic types. See \cite{MR788161} for the notions of projective representations of a finite groups and 
 see \cite{MR2099926}, \cite{MR572019}, and  \cite{MR486086} for results regarding the projective representations of finite Coxeter groups.

\begin{comment} 
\item Understanding modular reduction of complex irreducible representations of finite Coxeter groups is still quite open.
\end{comment} 

\end{enumerate}
{\bf Acknowledgments:} The author gratefully acknowledges the financial support from SERB, India, through grant SPG/2022/001099.
\bibliographystyle{authoryear} \bibliography{refs}

\end{document}